\documentclass[12pt]{article}
\usepackage{amsmath,amssymb}
\usepackage{color,graphicx}
\usepackage{ulem}

\newcommand{\bu}{\boldsymbol u}

\newtheorem{Theorem}{Theorem}

\newtheorem{lema}{Lemma}

\newcounter{remark}
\def\theremark {\arabic{remark}}
\newenvironment{remark}{\refstepcounter{remark}\par\noindent{\bf Remark\ \theremark}\ }{\par}
\newtheorem{Proof}{Proof}

\newenvironment{proof}{\begin{Proof}\rm}{\hfill $\Box$ \end{Proof}}

\usepackage{hyperref}
\usepackage[all]{hypcap}

\title{Optimal bounds for numerical approximations of finite horizon problems based on dynamic programming approach}

\author{Javier de Frutos\thanks{Instituto de Investigaci\'on en Matem\'aticas (IMUVA), Universidad
de Valladolid, Spain. Research supported
by Spanish MINECO
under grant PID2022-136550ND-I00 (frutos@mac.uva.es)}
  \and Julia Novo\thanks{Departamento de
Matem\'aticas, Universidad Aut\'onoma de Madrid, Spain.  Research supported
by Spanish MINECO
under grant PID2022-136550ND-I00. (julia.novo@uam.es)}}
\date{\today}
\begin{document}
\maketitle
\abstract{In this paper we provide optimal bounds for fully discrete approximations to finite horizon problems via  dynamic programming. We adapt the error analysis in \cite{nos} for the infinite horizon case to the
 finite horizon case. We prove an a priori bound of size $O(h+k)$ for the method, $h$ being the time discretization step and $k$ the spatial mesh size. Arguing with piecewise constants controls we are able to obtain first order of convergence in time and space under standard regularity assumptions, avoiding the more restrictive regularity assumptions on the controls required in \cite{nos}. 
 We show that the loss in the rate of convergence in time of the
 infinite case (obtained arguing with piece-wise controls)
   can be avoided in the finite horizon case}
\bigskip


\section{Introduction}
The numerical approximation of optimal control problems is of importance for many applications. In this paper we consider the dynamic programming  approach for solving finite horizon problems.  The dynamic programming principle gives a characterization of the value function as the unique viscosity solution of a nonlinear partial differential equation, the Hamilton-Jacobi-Bellman (HJB) equation. The value function is then used to get a synthesis of a feedback control law.  

In the present paper our concern is to give optimal error bounds for a fully discrete semi-Lagrangian method approaching the value function. For a method with a positive time step size $h$ and spatial elements of size $k$ we prove an optimal error bound of size $O(h+k)$ which gives first order of convergence in time and space for the method. We introduce a characterization of the fully discrete method inherited from \cite{nos}, see also \cite{nos_cor}. The temporal component of the error comes from the approximation of the dynamics by a discrete one based on the Euler method plus the approximation of the time integral by the composite rectangle rule. The spatial component of the error comes from the substitution of the functions, both in the dynamics and in the cost, by  piece-wise linear interpolants in space. Adapting the technique in \cite[Section 3.2]{nos}, based on piece-wise constant controls in time, we avoid making regularity assumptions on the controls. We only need a kind of discrete regularity assumption on the computed discrete controls to achieve the full order of convergence, see Remark 1. 

We think that the error analysis techniques shown in this paper are of interest to analyze similar methods applied to the same or analogous problems. To our knowledge, there are not many papers getting error estimates for methods solving finite horizon control problems.  In \cite{luca_et_al},
a  dynamic programming algorithm based on a tree structure (which does not require a spatial discretization of the problem) to mitigate the curse of dimensionality, see \cite{Alla_Fal_Luca}, is analyzed. First order bounds in time are obtained in the first part of this paper. In the second part, it is assumed that the continuous set of controls is replaced by a discrete set. The tree structure considers only spatial nodes that result of the discrete dynamics. To reduce the increase in the total number of nodes  a pruning criterion is applied in \cite{Alla_Fal_Luca} that replace a new node by and old one whenever the distance between them is small enough. In \cite{Alla_Fal_Luca} the pruning condition is too demanding since the difference between nodes is taken $O(h^2)$, $h$ being the time step. This fact comes from a factor $h$ that appears dividing in the error bounds. This problem can be solved with the error analysis presented in the present paper. The same problem was fixed with the error analysis of \cite{nos} in the infinite horizon case. Reference \cite{nos} is the first one  in which  a rate of convergence $O(h+k)$ is proved, improving the rate of convergence of size $O(k/h)$ shown in the literature, see \cite[Corollary 2.4]{falcone1}, \cite[Theorem 1.3]{falcone2}). 
 
Although the method analyzed in this paper does not avoid the curse of dimensionality one can apply reduced order techniques  to this method following \cite{nos_pod_control}. In \cite{nos_pod_control} a reduced order method based on proper orthogonal decomposition (POD) is applied for the numerical approximation of infinite horizon optimal control problems. The same ideas extend to finite horizon problems and will be subject of future research. 
For the error analysis of a reduced order method the error analysis of the present paper is essential.
Moreover, as stated before, the error analysis shown in this paper can be used to analyze or improve the analysis of   numerical methods for the same problem.

The outline of the paper is as follows. In Section 2 we introduce some notation. The fully discrete approximation is described in Section 3. In Section 4 we carry out the error analysis of the method. Finally, some interpolation arguments needed for the proof of the main theorem are included in the appendix.

\section{Model problem and preliminary results}
Throughout this section we follow the notation in \cite{Alla_Fal_Luca}.  Let us consider the system
\begin{eqnarray}\label{sys}
\dot y(s)&=&f(y(s),u(s),s),\quad s\in(t,T],\\
y(t)&=&x\in {\Bbb R}^d.\nonumber
\end{eqnarray}
We will denote by $y:[t,T]\rightarrow {\Bbb R}^d$ the solution, by $u$ the control $u:[t,T]\rightarrow {\Bbb R}^m$, by $f:{\Bbb R}^d\times {\Bbb R}^m\rightarrow {\Bbb R}^d$ the dynamics, and  by 
$$
{\cal U}=\left\{u:[t,T]\rightarrow U,\ {\rm measurable}\right\}
$$
the set of admissible controls where $U\subset {\Bbb R}^m$ is a compact set. We assume that there exists a unique solution for \eqref{sys} for each $u\in {\cal U}$.

The cost functional for the finite horizon optimal control problem will be given by
\begin{eqnarray}\label{cost}
J_{x,t}(u):=\int_t^T L(y(s,u),u(s),s)e^{-\lambda(s-t)}ds+g(y(T))e^{-\lambda(T-t)},
\end{eqnarray}
where $L: {\Bbb R}^d\times {\Bbb R}^m\times [t,T]\rightarrow {\Bbb R}$ is the running cost, $g:{\Bbb R}^d\rightarrow {\Bbb R}$ is the final cost, and $\lambda\ge 0$ is the discount factor.

The goal is to find a state-feedback control law $u(t)=\Phi(y(t),t)$, in terms of the state variable $y(t)$, where
$\Phi$ is the feedback map. To derive optimality conditions, dynamic programming principle (DPP) is used. The value
function for an initial condition is defined by
\begin{equation}\label{value}
v(x,t):=\inf_{u\in {\cal U}}J_{x,t}(u).
\end{equation}
The value function \eqref{value} satisfies the HHB equation for every $x\in {\Bbb R}^d$, $s\in[t,T)$:
\begin{eqnarray}\label{HJB}
&&-\frac{\partial v}{\partial s}(x,s)+\lambda v(x,s)+\max_{u\in U}\left\{-L(x,u,s)-\nabla v(x,s)\cdot f(x,u,s)\right\}=0,\\
&&v(x,T)=g(x).\nonumber
\end{eqnarray}
If the value function is known, then it  is possible to compute the optimal feedback control as
\begin{equation}\label{control}
u^*(t):={\rm arg \ max}_{u\in U}\left\{-L(x,u,t)-\nabla v(x,t)\cdot f(x,u,t)\right\}.
\end{equation}
Equation \eqref{HJB} is hard to solve analytically.  In next section we introduce a semi-Lagrangian method to approach the value function.

As in \cite{luca_et_al}, we assume that the functions $f$, $L$ and $g$ are continuous in all the variables and bounded:
\begin{eqnarray}\label{conti_f}
&&\|f(x,u,s)\|_\infty=\max_{1\le i\le n}|f_i(x,u,s)|\le M_f\\
&& |L(x,u,s)|\le M_L,\ |g(x)|\le M_g,\quad \forall x\in {\Bbb R}^d, u\in U,\ s\in[t,T].\label{conti_g}
\end{eqnarray}
We also assume that $f$ and $L$ are Lipschitz-continuous with respect to all the arguments:
\begin{eqnarray}\label{lip_f_1}
&&\|f(x,u,s)-f(y,u,s)\|_2\le L_f \|x-y\|_2,\quad \forall x,y\in {\Bbb R}^d, u\in U, s\in[t,T]\\
\label{lip_f_2}
&&\|f(x,u,s_1)-f(x,u,s_2)\|_2\le L_f |s_1-s_2|,\quad \forall x\in {\Bbb R}^d, u\in U, s_1,s_2\in[t,T],
\\
\label{lip_fu}
&&\|f(x,u_1,s)-f(x,u_2,s)\|_2\le L_f \|u_1-u_2\|_2,\quad \forall x\in {\Bbb R}^d, u_1,u_2\in U, s\in[t,T].
\end{eqnarray}
\begin{eqnarray}\label{lip_L_1}
&&|L(x,u,s)-L(y,u,s)\|_2\le L_L \|x-y\|_2,\quad \forall x,y\in {\Bbb R}^d, u\in U, s\in[t,T],\\
\label{lip_L_2}
&&|L(x,u,s_1)-L(x, u,s_2)\|_2\le L_L |s_1-s_2|,\quad \forall x\in {\Bbb R}^d, u\in U, s_1,s_2\in[t,T],\quad
\\
\label{lip_L_u}
&&|L(x,u_1,s)-L(x, u_2,s)\|_2\le L_L \|u_1-u_2\|_2,\quad \forall x\in {\Bbb R}^d, u_1,u_2\in U, s\in[t,T].
\end{eqnarray}
Finally, we assume that the cost $g$ is also Lipschitz-continuous:
\begin{equation}\label{lip_g}
|g(x)-g(y)|\le L_g \|x-y\|_2,\quad \forall x,y\in {\Bbb R}^d.
\end{equation}
\section{Fully discrete approximation}
Let us define $h=(T-t)/N$, where $N$ is the number of temporal steps. 
Let $\Omega$ a bounded polyhedron in $\Bbb R^d$ satisfying the following inward pointing condition on the dynamics for sufficiently small $h>0$
\begin{equation}\label{invariance}
y+hf(y,u)\in \overline \Omega,\quad \forall y \in \overline \Omega, u\in U.
\end{equation}
Let $\left\{S_j\right\}_{j=1}^{m_s}$ be a family of simplices which defines a regular triangulation of $\Omega$
$$
\overline \Omega=\bigcup_{j=1}^{m_s} S_j,\quad k=\max_{1\le j\le m_s}({\rm diam} \ S_j).
$$
We assume we have $n_s$ vertices/nodes $x_1,\ldots,x_{n_s}$ in the triangulation. Let $V^k$ be the space of piecewise affine functions from $\overline \Omega$ to ${\Bbb R}$ which are continuous in $\overline \Omega$ having constant gradients in the interior of any simplex $S_j$ of the triangulation. Then, a fully discrete scheme for the HJB equations is given by:
\begin{eqnarray}\label{fully_discrete}
&&v_{h,k}^n(x_i)=\min_{u\in U}\left\{hL(x_i,u,t_n)+\delta_hv_{h,k}^{n+1}(x_i+hf(x_i,u,t_n))\right\},\ n=N-1,\ldots,0\nonumber\\
&&v_{h,k}^{ N}(x_i)=g(x_i),\quad i=1,\ldots,n_s,
\end{eqnarray}
where $\delta_h=1-\lambda h$, $t_n=t+nh$. The functions $v_{h,k}^n$ are in $V^k$ and are defined by their values at the nodes $x_i$, $i=1,\ldots,n_s$ given by \eqref{fully_discrete}.

As in \cite{nos}, see also \cite{nos_cor}, we will give a characterization of the fully discrete method that allows to bound the error.

For any $x\in {\Bbb R^d}$, $n=0,\ldots N-1$ and $\bu^l_n=\left\{u_n^{l,n},\ldots,u_n^{l,N-1}\right\}$ with all its components $u_n^{l,k}\in U$, $n\le k\le N-1$, for $1\le l\le n_s$, let us define the fully discrete functional
\begin{eqnarray}\label{eq:funcional_fd}
J_{h,k}^n(x,\bu^1_n,\ldots,\bu^{n_s}_n)&:=&h\sum_{j=n}^{N-1} \delta_h^{j-n}I_k L(\hat y_j,u_n^{1,j},\ldots,u_n^{n_s,j},t_j)
+I_kg(\hat y_{N})e^{-\lambda(T-t_n)}, \ \\
\hat y_{j+1}&=&\hat y_j+h I_k f(\hat y_j,u_n^{1,j},\ldots,u_n^{n_s,j},t_j),\ \hat y_n=x,\ j=n,\ldots {N}-1,\quad\label{eq:eulerb}
\end{eqnarray}
where, using the barycentric coordinates, 
$$
x=\sum_{j=1}^{n_s}\mu_j(x) x_j, \ 0\le \mu_j(x)\le 1,\ \sum_{j=1}^{n_s}\mu_j(x)=1,
$$
for any $x$, any $t$ and any $u^1,\ldots, u^{n_s}\in U$ 
the interpolants $I_k L(x,u^{1},\ldots,u^{n_s},t)$ and  $I_k f(x,u^{1},\ldots,u^{n_s},t)$ are defined by
\begin{eqnarray}\label{interb}
&&I_k L(x,u^1,\ldots,u^{n_s},t)=\sum_{j=1}^{n_s}\mu_j(x) L(x_j,u^j,t),\\
&&I_k f(x,u^1,\ldots,u^{n_s},t)=\sum_{j=1}^{n_s}\mu_j(x) f(x_j,u^j,t),\nonumber
\end{eqnarray}
and $I_k g$ is the piecewise linear interpolant of $g$ in $V^k$.

Now, as in \cite{nos_cor}, for any $n=0,\ldots,N-1$, we define $w_{h,k}^n\in V^k$ by
\begin{eqnarray}\label{eq:cha}
w_{h,k}^n(x)=\inf_{\bu^1_n,\ldots, \bu^{n_s}_n}  J_{h,k}^n(x,\bu^1_n,\ldots,\bu^{n_s}_n).
\end{eqnarray}

Following \cite[Theorem 1]{nos_cor} it follows that
\begin{Theorem}
For any $x\in \overline \Omega$, $n=0,\ldots,N$, the function $w_{h,k}^n(x)$ satifies the equation

$$w_{h,k}^n(x)=\inf_{u_n^1,\dots,u_n^{n_s}}\{hI_kL(x,u_n^1,\dots,u_n^{n_s},t_n)+\delta_h w_{h,k}^{n+1}(x+hI_kf(x,u_n^1,\dots,u_n^{n_s},t_n))\}$$
and, as a consequence, for each node $x_i$, $i=1,\dots,n_s$

$$w_{h,k}^n(x_i)=\inf_{u_n^{i}}\left\{hL(x_i,u^{i}_n,t_n)+\delta_h w_{h,k}^{n+1}(x_i+hf(x_i,u^{i}_n,t_n))\right\}.$$
\end{Theorem}
Following \cite[Theorem 2]{nos_cor} we also have the following characterization
\begin{Theorem}\label{th:acla}
For each node $x_i$, $i=1,\dots,n_s,$
let us denote by $u^{i}_n$ the argument giving the minimum in
$$w_{h,k}^n(x_i)=\inf_{u_n^i}\left\{hL(x_i,u^{i}_n,t_n)+\delta_h w_{h,k}^{n+1}(x_i+hf(x_i,u^{i}_n,t_n))\right\}.$$
Then, for any $x\in \overline \Omega$ the subsets of controls $\bu_n^i$, $i=1,\ldots,n_s$ giving the minimum in
$$w_{h,k}^n(x)=\inf_{\bu_n^1,\ldots, \bu_n^{n_s}}   J_{h,k}^n(x,\bu^1_n,\ldots,\bu^{n_s}_n)$$
are determined by the values of the controls at the nodes, $u^i_n$, $1\le i\le n_s$. More precisely, for any $j$, the values $\left\{u_n^{1,j},\ldots,u_n^{n_s,j}\right\}$  in \eqref{eq:funcional_fd}-\eqref{eq:eulerb}, $n\le j\le N-1$, are $u_n^{i,j}=u^i_n$ if $\mu_i(\hat y_j)\neq 0$ and $u_n^{i,j}=0$ if $\mu_i(\hat y_j)=0.$ 
\end{Theorem}
From the above theorem is clear that the interpolants defined in \eqref{interb} are always piecewise functions in $V^k$.  Then,
it is immediate to prove $w_{h,k}^n\in V^k$ which implies $w_{h,k}^n$ is the unique solution defined by \eqref{fully_discrete} i.e., $w_{h,k}^n=v_{h,k}^n$ and, as a consequence, gives a characterization of the fully discrete functional. This characterization is used in the present paper to prove the error bounds of the method.
\section{Error analysis of the method}
We follow the error analysis of \cite[Section 3.2]{nos}, based on the analysis in \cite{boba}, to prove convergence of the method arguing with piecewise constant controls.

Let us denote
$$
{\cal U}^{pc}=\left\{u\in {\cal U}\ \mid\ u(t)=u_l,\ t\in [t_l,t_{l+1}), \ 0\le l\le N-1\right\},
$$
with $u_l$ constant. Let us observe that we can consider the continuous problem for controls in ${\cal U}^{pc}$.

The following lemma follows the error analysis in \cite[Lemma 1]{nos}, see also \cite[Lemma 1.2, Chapter VI]{Bardi}.
Along the error analysis $C$ will represent a generic constant that is not always necessarily the same and that does not depend neither on $h$ nor in $k$.
\begin{lema}\label{lema1} Let $0\le n\le N-1$, $x\in {\Bbb R}^d\subset \overline \Omega$ and $J_{x,t_n}(u)$, $ J_{h,k}^n(x,\bu^1_n,\ldots,\bu^{n_s}_n)$
the functionals defined in \eqref{cost} and \eqref{eq:funcional_fd}-\eqref{eq:eulerb}, respectively. Assume  conditions \eqref{conti_f}, \eqref{conti_g}, \eqref{lip_f_1}, \eqref{lip_f_2}, \eqref{lip_L_1}, \eqref{lip_L_2} and \eqref{lip_g}   hold. Then
\begin{eqnarray}\label{eq_lim0_cons}
| J_{x,t_n}(u)- J_{h,k}^n(x,\bu^1_n,\ldots,\bu^{n_s}_n)|\le C(h+k),
\end{eqnarray}
where $u\in {\cal U}^{pc}$ and for $i=1,\ldots,n_s$, $\bu_n^i=\bu=\left\{u_n,u_{n+1},\ldots,u_{N-1}\right\}$ with $u_l=u(t),$ $t\in [t_l,t_{l+1})$, $l=n,\ldots,N-1$.
\end{lema}
\begin{proof}
 Let $y(s)$, $s\in[t_n,T]$, be the solution of \eqref{sys} and let us denote by $\tilde y(s)=\hat y_l$,
$l=[s/h]$ where $\hat y_l$ is the solution of \eqref{eq:eulerb}. Then, $\tilde y$ can be expressed as
$$
\tilde y(s)=x+\int_{t_n}^{[s/h]h}I_{k}f(\tilde y(\tau),u(\tau),[\tau/h]h)~d\tau.
$$
And,
\begin{eqnarray*}
y(s)-\tilde y(s)&=&\int_{t_n}^{[s/h]h}\left(f(y(\tau),u(\tau),\tau)-I_{k}f(\tilde y(\tau), u(\tau),[\tau/h]h)\right)~d\tau\\
&&\quad+\int_{[s/h]h}^s f(y(\tau),u(\tau),\tau)~d\tau.
\end{eqnarray*}
From the above equation, applying \eqref{conti_f}, we get
\begin{eqnarray}\label{eq:infty1}
\|y(s)-\tilde y(s)\|_\infty\le \int_{t_n}^{[s/h]h}\|f(y(\tau),u(\tau),\tau)-I_{k}f(\tilde y(\tau),\overline u(\tau),[\tau/h]h)\|_\infty~d\tau +M_f h.
\end{eqnarray}
Let us bound now the term inside the integral.
Adding and subtracting terms we get
\begin{align}\label{eq:decom}
\|f(y(\tau),u(\tau),\tau)&-I_{k}f(\tilde y(\tau), u(\tau),[\tau/h]h)\|_\infty\le\\
&
\quad\quad\, \|f(y(\tau),u(\tau),\tau)-f(y(\tau),u(\tau),[\tau/h]h)\|_\infty\nonumber\\
&\quad+
\|f(y(\tau), u(\tau),[\tau/h]h)-I_{k}f(y(\tau), u(\tau),[\tau/h]h)\|_\infty\quad
\nonumber\\
&\quad +\|I_{k}f(y(\tau), u(\tau),[\tau/h]h)-I_{k}f(\tilde y(\tau), u(\tau),[\tau/h]h)\|_\infty.\nonumber
\end{align}
For the first term on the right-hand side, applying \eqref{lip_f_2}, it is easy to prove that
\begin{eqnarray}
\label{eu_cua}
\int_{t_n}^{[s/h]h}\|f(y(\tau),u(\tau),\tau)-f(y(\tau),u(\tau),[\tau/h]h)\|_\infty d\tau\le C h,\quad C=(T-t_n)L_f.
\end{eqnarray}
To bound  the other terms on the right-hand side of \eqref{eq:decom} arguing, as in \cite{Alla_Falcone_Volkwein}, we observe that
for any $y\in  \overline \Omega$ there exists an index $l$ with $y\in \overline S_l\subset \overline \Omega$. Let
us denote by $J_l$ the index subset such that $x_i\in S_l$ for $i\in J_l$. Writing
$$
y=\sum_{i=1}^{n_S} \mu_i x_i,\quad 0\le \mu_i\le 1,\quad \sum_{i=1}^{n_S} \mu_i=1,
$$
it is clear that $\mu_i=0$ holds for any $i\not \in J_l$. Now, we observe that for any $u\in {\cal U}$ and
any time $s$,
applying \eqref{lip_f_1} we get for $1\le j\le d$
\begin{eqnarray}\label{needed}
|f_j(y,u,s)-I_{k}f_j(y,u,s)|&=&|\sum_{i=1}^{n_S}\mu_i f_j(y,u,s)-\sum_{i=1}^{n_S}\mu_i I_{k}f_j(x_i,u,s)|\nonumber\\
&=&|\sum_{i\in J_l}\mu_i(f_j(y,u,s)-f_j(x_i,u,s)|\nonumber\\
&\le&\sum_{i\in J_l}\mu_i L_f \|y-x_i\|_2\le L_f k,
\end{eqnarray}
where in the last inequality we have applied $\|y-x_i\|_2\le k$, for $i\in J_l$. From the above inequality we get for the second term
on the right-hand side of \eqref{eq:decom}
\begin{eqnarray}\label{eq:segundo}
\int_{t_n}^{[s/h]h}\|f(y(\tau),u(\tau),[\tau/h]h)-I_kf( y(\tau),\overline u(\tau),[\tau/h]h)\|_\infty\le C k.
\end{eqnarray}
For the third term on the right-hand side of \eqref{eq:decom} we observe that the difference of the interpolation operator evaluated at
two different points can be bounded in terms of the constant gradient of the interpolant in the element to which those points belong times
the difference of them, i.e.,
$$
I_{k}f_j(y,u,s)-I_{k}f_j(\tilde y,u,s)=\nabla I_{k} f_j(\tilde y,u,s)\cdot (y-\tilde y)\le \|\nabla I_{k} f_j(\tilde y,u,s)\|_2\|y-\tilde y\|_2.
$$
Moreover, $\nabla I_{k}f_j$ can be bounded in terms of the lipschitz constant of $f$, $L_f$, more precisely,  $\|\nabla I_{k} f_j(\tilde y,u,s)\|_2\le C\sqrt{d}L_f$. Then,
$$
|I_{k}f_j(y,u,s)-I_{k}f_j(\tilde y,u,s)|\le C L_f \sqrt{d}\|y-\tilde y\|_2.
$$
As a consequence, for
the third term on the right-hand side of \eqref{eq:decom} we get
\begin{eqnarray}\label{eq:tercero}
&&\int_{t_n}^{[s/h]h}\|I_{k}f(y(\tau),u(\tau),[s/h]h)-I_{k}f(\tilde y(\tau),u(\tau),[s/h]h)\|_\infty\le\\
&&\quad C \sqrt{d}L_f
\int_{t_n}^s\|y(\tau)-\tilde y(\tau)\|_2\le C {d}L_f\int_{t_n}^s\|y(\tau)-\tilde y(\tau)\|_\infty.\nonumber
\end{eqnarray}
From \eqref{eu_cua}, \eqref{eq:segundo} and \eqref{eq:tercero}  we get for
$\overline C=CdL_f$
\begin{eqnarray*}
\|y(s)-\tilde y(s)\|_\infty\le \overline C \int_{t_n}^s \|y(\tau)-\tilde y(\tau)\|_\infty~d\tau+ C(h+k).
\end{eqnarray*}
Applying Gronwall's lemma we obtain
\begin{eqnarray}\label{ymenosytilde}
\|y(s)-\tilde y(s)\|_\infty\le \frac{e^{\overline C s}}{\overline C}C(h+k)\le C(h+k),\quad s\in[t_n,T].
\end{eqnarray}
For simplicity, in the sequel we will denote by
$$
J_{h,k}^n(x,\bu)=J_{h,k}^n(x,\bu_n^1,\ldots,\bu_n^{n_s}).
$$
We decompose
\begin{eqnarray}\label{eq:tb}
| J_{x,t_n}(u)-J^n_{h,k}(x,\bu)|\le X_1+X_2,
\end{eqnarray}
with
\begin{eqnarray*}
X_1&=&\left|h\sum_{j=n}^{N-1} \delta_h^{j-n} I_{k}L(\hat y_j,u_j,t_j)-\int_0^T L(y(s),u(s),s)e^{-\lambda (s-t_n)}~ds\right|,\\
X_2&=&\left|(g(y(T))-I_kg(\tilde y(T))e^{-\lambda(T-t_n)}\right|.
\end{eqnarray*}
We start bounding the second term. Adding and subtracting terms and using \eqref{lip_g} 
\begin{eqnarray*}
\left|g(y(T))-I_kg(\tilde y(T)\right|&\le& \left|g(y(T))-g(\tilde y(T))\right|+\left|g(\tilde y(T))-I_kg(\tilde y(T))\right|\nonumber\\ &\le& L_g \|y(T)-\tilde y(T)\|_2+\left|g(\tilde y(T))-I_kg(\tilde y(T))\right|.
\end{eqnarray*}
To bound the first term on the right-hand side above we apply \eqref{ymenosytilde} to get
$$
 L_g\|y(T)-\tilde y(T)\|_2 \le C (k+h).
$$
To bound the second term we apply \cite{paper_lip} where it is proved that the error in the piece-wise linear interpolant of a Lipschitz function can be bounded in terms of the lipschitz constant of that function so that
$$
\left|g(\tilde y(T))-I_kg(\tilde y(T))\right|\le C(L_g)k.
$$
Then,
\begin{eqnarray}\label{X2}
X_2\le e^{-\lambda(T-t_n)}C(h+k).
\end{eqnarray}
To bound the first integral term we observe that
$$
X_1=\left|\int_{t_n}^{T}\left(I_{k}L(\tilde y(s), u(s),[s/h]h)\delta_h^{[s/h]}-L(y(s),u(s),s)e^{-\lambda s}\right)~ds\right|.
$$
Then, we can write
\begin{eqnarray}\label{cota_x1}
X_1&\le& X_{1,1}+X_{1,2}+X_{1,3}+X_{1,4}\\
&:=& \left|\int_{t_n}^{T} I_{k}L(\tilde y(s), u(s),[s/h]h)(\delta_h^{[s/h]-n}- e^{-\lambda (s-t_n)})~ds\right|\nonumber\\
&&\quad+\left|\int_{t_n}^{T} \left(I_{k}L(\tilde y(s), u(s),[s/h])-I_{k}L(y(s), u(s),[s/h])\right)e^{-\lambda (s-t_n)}~ds\right|
\nonumber\\
&&\quad+\left|\int_{t_n}^{T} \left(I_{k}L(y(s), u(s),[s/h]h)-L(y(s), u(s),[s/h]h)\right)e^{-\lambda (s-t_n)}~ds\right|\nonumber\\
&&\quad +\left|\int_{t_n}^{T} \left(L(y(s), u(s),[s/h]h)-L(y(s), u(s),s)\right)e^{-\lambda (s-t_n)}.~ds\right|\nonumber\end{eqnarray}
Now we bound the four terms on the right-hand side of \eqref{cota_x1}. To bound the first term we will apply  
$|I_{k}L(y,u,s)|\le |L(\cdot,u,s)|_\infty\le M_L$ (see \eqref{conti_g}) to obtain
\begin{eqnarray*}
X_{1,1}=M_L\int_{t_n}^{T} |\delta_h^{[s/h]-n}- e^{-\lambda (s-t_n)}|~ds.
\end{eqnarray*}
Now we write $\delta_h^{[s/h]-n}=e^{-\lambda \theta ([s/h]-n)h}$, for $\theta=-\log(\delta_h)/(\lambda h)$. Applying the mean value theorem to the function $e^{-\lambda (s-t_n)}$ and taking into account that since $([s/h]-n)h\le s-t_n\le ([s/h]-n)h+h$ then $|s-t_n\theta ([s/h]-n)h|\le (\theta-1)(T-t_n)+\theta h$
and that $\theta- 1=O(h)$ then we get
\begin{eqnarray}\label{cota_x1_1}
X_{1,1} \le M_L(T-t_n) \lambda((\theta-1)(T-t_n)+\theta h)\le C h.
\end{eqnarray}
To bound the next term we argue as in \eqref{eq:tercero} and use \eqref{ymenosytilde} to get
\begin{eqnarray*}
X_{1,2}\le{CL_L}\int_{t_n}^T  \|y(s)-\tilde y(s)\|_\infty
e^{-\lambda (s-t_n)}~ds\le C(h+k).
\end{eqnarray*}
For the third term, arguing as in \eqref{needed} we get
\begin{eqnarray}\label{cota_x1_3}
X_{1,3}\le L_L k \int_{t_n}^T e^{-\lambda (s-t_n)}~ds\le C k.
\end{eqnarray}
Finally the last term $X_{1,4}=O(h)$ applying \eqref{lip_L_2} and arguing as in \eqref{eu_cua}.
\end{proof}
Let us observe that for any $x\in \overline \Omega$
$$
|v_{h,k}^N(x)-v(x,T)|=|I_kg(x)-g(x)|\le Ck.
$$
In next theorem we bound the difference $|v_{h,k}^n(x)-v(x,t_n)|$ for $n=0,\ldots,N-1$.
As in \cite[Theorem 7]{nos}  for the proof we need to assume an additional convexity assumption, see \cite[(A4)]{boba},
\begin{itemize}
\item (CA) For every $y\in {\Bbb R}^d$, and any $s\in[t,T]$
\begin{eqnarray*}
\left\{f(y,u,s), L(y,u,s),\quad  u\in U\right\}
\end{eqnarray*}
is a convex subset of ${\Bbb R}^{d+1}$.
\end{itemize}
We also need to assume \eqref{lip_fu} and \eqref{lip_L_u} for the proof that can be found in the appendix.
\begin{Theorem}\label{th_4_cons}
Assume  conditions \eqref{conti_f}, \eqref{conti_g}, \eqref{lip_f_1}, \eqref{lip_f_2}, \eqref{lip_fu}, \eqref{lip_L_1}, \eqref{lip_L_2}, \eqref{lip_L_u}, \eqref{lip_g} and {\rm(CA)}  hold.
For $0\le n\le N-1$ and $x\in \overline \Omega$ there exist constants $C_1$ and $C_2$, independent of $h$ and $k$, such that the following bound holds 
\begin{equation}\label{cota_buena_cons}
|v_{h,k}^n(x)-v(x,t_n)|\le C_1h+C_2k.
\end{equation}
The constant $C_2$ depends linearly on
\begin{equation}\label{laLu}
L_u=\max \frac{\|u_n^i-u_n^j\|_2}{|x_i-x_j|},\quad 1\le i,j\le n_S,
\end{equation}
where $u_n^i$ are the arguments giving the minimum in $v_{h,k}^n(x_i)$ in \eqref{fully_discrete}.
\end{Theorem}
\begin{proof}
In view of \eqref{eq:cha}
let us denote by $\left\{\bu_n^1,\ldots, \bu_n^{n_s}\right\}$
a control giving the minimum
$$
v_{h,k}^n(x)= J_{h,k}^n(x, \bu_n^1,\ldots,\bu_n^{n_s}),
$$
with $\left\{u_n^1,\ldots,u_n^{n_s}\right\}$, as stated in the assumptions of the present theorem, the controls giving the minimum at the nodes.

Let us define $\bu=\left\{u_n,\ldots,u_{N-1}\right\}$ where for $l=n,\ldots,N-1$, 
and  $\hat y_l$ (defined by \eqref{eq:eulerb}) written as  
$\hat y_l=\sum_{j=1}^{n_s}{\mu_j(\hat y_l)}x_j$
then
$u_l=\sum_{j=1}^{n_s}{\mu_j(\hat y_l)} u_n^j$.
Let us denote by
$$
 J_{h,k}^n(x,\bu)= J_{h,k}^n(x,\bu,\ldots,\bu).
$$
Then
$$
v(x,t_n)-v_{h,k}^n(x)\le \left(J_{x,t_n}(u)- J_{h,k}^n(x,\bu)\right)+\left(J_{h,k}^n(x,\bu)-J_{h,k}^n(x,\bu_n^1,\ldots,\bu_n^{n_s})\right),
$$
where $u\in{\cal U}^{pc}$ such that $u(t)= u_l$, $t\in[t_l,t_{l+1})$, $n\le l\le N-1$.
 Applying Lemma \ref{lema1} to bound the first term and  applying standard interpolation arguments for the second one (see \eqref{prin_apen}  in the appendix), there exists a  positive constant   such that
\begin{equation}\label{puessi}
v(x,t_n)-v_{h,k}^n(x)\le C_1 h+C_2k.
\end{equation}
Now, we need to bound $v_{h,k}^n(x)-v(x,t_n)$. Let us denote by $\underline u \in U$ the control giving the minimum in \eqref{cost} for $t=t_n$.
so that
\begin{equation}\label{fal_0}
v(x,t_n)= J_{x,t_n}(\underline u)=\int_{t_n}^T L(y(s,u),\underline u(s),s)e^{-\lambda(s-t_n)}ds+g(y(T))e^{-\lambda(T-t_n)}
\end{equation}
The following argument uses the ideas from \cite[Appendix B]{boba} and \cite[Theorem 7]{nos}. 

For any $t_l$ we can write
\begin{equation}\label{tan}
y(t)=y(t_l)+\int_{t_l}^t f(y(s),\underline u(s),s)ds.
\end{equation}
Applying \eqref{conti_f}
\begin{equation}\label{obvi}
\|y(t)-y(t_l)\|_\infty \le M_f h,\quad t\in[t_l,t_{l+1}].
\end{equation}
Then, for any $t\in [t_l,t_{l+1}]$, using the above inequality and \eqref{lip_f_1}
 we obtain
$$
\left\|\int_{t_l}^t f(y(s),\underline u(s),s)-f(y(t_l),\underline u(s),s) ds\right\|_\infty\le \sqrt{d}L_f M_f h^2.
$$
As a consequence, we get
\begin{equation}\label{fal_1}
y(t)=y(t_l)+\int_{t_l}^t f(y(t_l),\underline u(s),s)ds+\epsilon_k,\quad \|\epsilon_l\|_\infty\le \sqrt{d} L_f M_f h^2.
\end{equation}
On the other hand, as in \cite[(B.6a), (B.6b)]{boba}, \cite[Theorem 7]{nos}, thanks to (CA), for any $l$, there exists $\underline u_l$ such that
\begin{eqnarray}
\int_{t_l}^{t_{l+1}}f(y(t_l),\underline u(s),t_l)ds=h f(y(t_l),\underline u_l,t_l)\label{mean_f}\\
\int_{t_l}^{t_{l+1}}L(y(t_l),\underline u(s),t_l)ds= h L(y(t_l),\underline u_l,t_l)\label{mean_g_i}.
\end{eqnarray}
Let us also observe that, applying \eqref{mean_f} and \eqref{lip_f_2}
\begin{eqnarray}\label{nueva_s1}
&&\int_{t_l}^{t_{l+1}}(f(y(t_l),\underline u(s),s)-f(y(t_l),\underline u_l,s))ds
\nonumber\\
&&=\int_{t_l}^{t_{l+1}}(f(y(t_l),\underline u(s),s)-f(y(t_l),\underline u(s),t_l))ds\nonumber\\
&&\quad+\int_{t_l}^{t_{l+1}}(f(y(t_l),\underline u(s),t_l)-f(y(t_l),\underline u_l,t_l))ds\nonumber\\
&&\quad+\int_{t_l}^{t_{l+1}}(f(y(t_l),\underline u_l,t_l)-f(y(t_l),\underline u_l,s))ds\le 2L_f h^2.
\end{eqnarray}
From \eqref{mean_g_i} and \eqref{conti_g}  we get
\begin{eqnarray}\label{mean_g}
&&\int_{t_l}^{t_{l+1}}(L(y(t_l),\underline u_l,s)-L(y(t_l),\underline u(s),s))e^{-\lambda (s-t_n)}ds\le \nonumber\\
&&\quad\int_{t_l}^{t_{l+1}}L(y(t_l),\underline u_l,s)(e^{-\lambda (s-t_n)}-e^{-\lambda (t_l-t_n)})ds
\nonumber\\
&&\quad+\int_{t_l}^{t_{l+1}}(L(y(t_l),\underline u_l,s)-L(y(t_l),\underline u(s),t_l)e^{-\lambda (t_l-t_n)}ds
\nonumber\\
&&\quad+\int_{t_l}^{t_{l+1}}L(y(t_l),\underline u(s),t_l)(e^{-\lambda (t_l-t_n)}-e^{-\lambda (s-t_n)})ds
\nonumber\\
&&\le 2\lambda M_Lh^2= C h^2,\ C=2\lambda M_L.
\end{eqnarray}
From \eqref{fal_1} and \eqref{nueva_s1}
\begin{eqnarray}\label{lauso}
y(t_{l+1})&=& y(t_l)+\int_{t_l}^{t_{l+1}} f(y(t_l),\underline u(s),s)ds+\epsilon_l=y(t_l)+\int_{t_l}^{t_{l+1}} f(y(t_l),\underline u_l,s)ds\nonumber\\
&&\quad +\int_{t_l}^{t_{l+1}} (f(y(t_l),\underline u(s),s)ds-f(y(t_l),\underline u_l,s))ds+\epsilon_l\le\nonumber\\
&&y(t_l)+\int_{t_l}^{t_{l+1}} f(y(t_l),\underline u_l,s)ds+\epsilon_l+2L_f h^2.
\end{eqnarray}
Let us denote by $y^{pc}$ the time-continuous trayectory solution with the same initial condition as $y$ associated
to the control $
\underline u^{pc}(t)=\underline u_l,\ \forall t\in[t_l,t_{l+1}),\  l=n,\ldots,N-1$.

Arguing as in \eqref{fal_1}  we get
\begin{equation}\label{fal_2}
y^{pc}(t)=y^{pc}(t_l)+\int_{t_l}^t f(y^{pc}(t_l),\underline u_l,s)ds+\epsilon_l,\quad \|\epsilon_l\|_\infty\le \sqrt{d}L_f M_f h^2.
\end{equation}
Subtracting \eqref{fal_2} from \eqref{lauso} (applied with $l-1$) and  using \eqref{lip_f_1} we obtain
\begin{eqnarray*}
\|y(t_l)-y^{pc}(t_l)\|_\infty&\le &\|y(t_{l-1})-y^{pc}(t_{l-1})\|_\infty+\sqrt{d}hL_f\|y(t_{l-1})-y^{pc}(t_{l-1})\|_\infty\\
&&\quad +2\|\epsilon_{l-1}\|_\infty+2L_f h^2\\
&\le&(1+ h\sqrt{d}L_f)\|y(t_{l-1})-y^{pc}(t_{l-1})\|_\infty+C h^2,
\end{eqnarray*}
where $C=2\sqrt{d}L_f M_f+2L_f$.
Since $y(y_n)=y^{pc}(t_n)=x$ by standard recursion we get
\begin{equation}\label{fal_3}
\|y(t_l)-y^{pc}(t_l)\|_\infty\le (1+ h\sqrt{d}L_f)^{N-n}C(N-n)h^2\le e^{\sqrt{d}h(N-n)L_f}C(N-n)h^2\le C h,\end{equation}
for $C=C T e^{\sqrt{d}TL_f}$.

Now, for the control $\underline u \in{\cal U}$ giving the minimum in \eqref{cost} for $t=t_n$ and for
$
\underline \bu=\left\{\underline u_n,\ldots \underline u_{N-1}\right\},
$
denoting by
$$
J_{h,k}^n(x,\underline \bu)= J_{h,k}^n(x,\underline \bu,\ldots,\underline \bu).
$$
we obtain
\begin{eqnarray*}
v_{h,k}^n(x)-v(x,t_n)&\le& J_{h,k}^n(y,\underline \bu)-J_{x,t_n}(\underline u)=
 J_{h,k}(y,\underline \bu)- J_{x,t_n}(\underline u^{pc})\\
 &&\quad+J_{x,t_n}(\underline u^{pc})- J_{x,t_n}(\underline u).
\end{eqnarray*}
The first term on the right-hand side above is bounded in Lemma \ref{lema1} so that
\begin{eqnarray*}
v_{h,k}^n(x)-v(x,t_n)\le C(h+k)+J_{x,t_n}(\underline u^{pc})- J_{x,t_n}(\underline u).
\end{eqnarray*}
To conclude we need to bound the second term.
We write
\begin{eqnarray}\label{fal_6}
J_{x,t_n}(\underline u^{pc})- J_{x,t_n}(\underline u)&=&
\int_{t_n}^T \left(L(y^{pc}(s),\underline u^{pc}(s),s)-L(y(s),\underline u(s),s)\right)e^{-\lambda (s-t_n)}ds\nonumber\\
&&\quad+e^{-\lambda(T-t_n)}(g(y^{pc}(T))-g(y(T)).
\end{eqnarray}
For the second term on the right-hand side above, applying \eqref{lip_g} we get
$$
|e^{-\lambda(T-t_n)}(g(y^{pc}(T))-g(y(T))|\le C L_g \|y^{pc}(T)-y(T)\|_2\le C h,
$$
where in the last inequality we have applied \eqref{fal_3} with $l=N$.
To conclude we will bound the first term on the right-hand side of \eqref{fal_6}.
We observe that
\begin{eqnarray*}
&&\int_{t_n}^T \left(L(y^{pc}(s),\underline u^{pc}(s),s)-L(y(s),\underline u(s),s)\right)e^{-\lambda (s-t_n)}ds=
\\
&&\quad \sum_{l=n}^{N-1}\int_{t_l}^{t_{l+1}}\left(L(y^{pc}(s),\underline u_l,s)-L(y(s),\underline u(s),s)\right)e^{-\lambda (s-t_n)}ds.
\end{eqnarray*}
Adding and subtracting terms we get
\begin{eqnarray*}
&&\int_{t_l}^{t_{l+1}}\left(L(y^{pc}(s),\underline u_l,s)-L(y(s),\underline u(s),s)\right)e^{-\lambda (s-t_n)}ds
=\\
&&\quad \int_{t_l}^{t_{l+1}}\left(L(y^{pc}(s),\underline u_l,s)-L(y(s),\underline u_l,s)\right)e^{-\lambda (s-t_n)}ds
\\
&&\quad +\int_{t_l}^{t_{l+1}}\left(L(y(s),\underline u_l,s)-L(y(t_l),\underline u_l,s)\right)e^{-\lambda (s-t_n)}ds\\
&&\quad +\int_{t_l}^{t_{l+1}}\left(L(y(t_l),\underline u_l,s)-L(y(t_l),\underline u(s),s)\right)e^{-\lambda (s-t_n)}ds
\\
 &&\quad +\int_{t_l}^{t_{l+1}}\left(L(y(t_l),\underline u(s),s)-L(y(s),\underline u(s),s)\right)e^{-\lambda (s-t_n)}ds
\end{eqnarray*}
Applying \eqref{lip_L_1} and \eqref{mean_g} we get
\begin{eqnarray*}
&&\int_{t_l}^{t_{l+1}}\left(L(y^{pc}(s),\underline u_l,s)-L(y(s),\underline u(s),s)\right)e^{-\lambda (s-t_n)}ds
\le\\
&&L_L\int_{t_l}^{t_l+1}(\|y^{pc}(s)-y(s)\|_2+2\|y(s)-y(t_l)\|_2)e^{-\lambda (s-t_n)}ds+C h^2.
\end{eqnarray*}
Taking into account the following decomposition
$$
\|y^{pc}(s)-y(s)\|_2\le \|y^{pc}(s)-y^{pc}(t_l)\|_2+\|y^{pc}(t_l)-y(t_l)\|_2+\|y(t_l)-y(s)\|_2,
$$
and applying \eqref{obvi}  (than can also be applied to $y^{pc}$)
and \eqref{fal_3}
we reach
\begin{eqnarray*}
\int_{t_l}^{t_{l+1}}\left(L(y^{pc}(s),\underline u_l,s)-L(y(s),\underline u(s),s)\right)e^{-\lambda (s-t_n)}ds\le C h^2.
\end{eqnarray*}
And then
\begin{eqnarray*}
\int_{t_n}^T \left(L(y^{pc}(s),\underline u^{pc}(s),s)-L(y(s),\underline u(s),s)\right)e^{-\lambda (s-t_n)}ds\le Ch.
\end{eqnarray*}
As a consequence, we finally obtain
\begin{eqnarray*}
v_{h,k}^n(x)-v(x,t_n)\le C(h+k),
\end{eqnarray*}
which finish the proof.
\end{proof}
\begin{remark}
We observe that the term $h^{1/(1+\beta)}$ in \cite[Theorem 7]{nos} (and, as a consequence, the reduction of the error in time in the rate of convergence of the method) comes from having an infinite horizon problem. In the finite horizon case we handle in this paper there is no reduction in the rate of convergence under assumption {\rm(CA)}.

The linear dependence of the constant $C_2$ in \eqref{cota_buena_cons} on $L_u$ comes from the interpolation arguments in the appendix. It is clear that since we have a finite number of terms in \eqref{laLu} $L_u$ is always a finite constant for any $h$, $k$. However, we cannot prove that is independent on both $h$ and $k$. On the other hand, in practice one can always compute the value of $L_u$ in a numerical experiment. In case, $L_u$ behaves for example as $k^{-\alpha}$ with $0<\alpha<1$ we still achieve convergence of order $k^{1-\alpha}$ in space. It is also clear that to be able to obtain the full order (first order of convergence in time and space) we need a constant $L_u$ independent of $h$ and $k$. This assumption  is like a discrete Lipchitz-continuity condition on the nodal values of the computed discrete controls. As can be seen in the proof of the appendix, Lipschit-continuity of the functions is the minimal requirement to achieve the full order of convergence (first order) in the piece-wise linear interpolation, see \cite{paper_lip}.
\end{remark}
\appendix \label{app}

\section{Interpolation bounds}
\begin{lema}
For any $x\in  \overline \Omega$ 
let us denote by $\left\{\bu_n^1,\ldots,\bu_n^{n_s}\right\}$, $\bu\in {\cal U}$ the control giving the minimum
$$
v_{h,k}^n(x)= J_{h,k}^n(x, \bu_n^1,\ldots,\bu_n^{n_s}),
$$
with $\left\{ u^1_n,\ldots, u_n^{n_s}\right\}$ the controls giving the minimum at the nodes.

Let $\bu^j_n=\left\{u_n^{j,n},\ldots,u_n^{j,N-1}\right\}$, $1\le j\le n_S$.
Let us define $ \bu=\left\{ u_n,\ldots,u_{N-1}\right\}$ as follows. For  
  $\hat y_l$ (defined by \eqref{eq:eulerb}) written as
$\hat y_l=\sum_{j=1}^{n_s}{\mu_j(\hat y_l)}x_j$
then
$ u_l=\sum_{j=1}^{n_s}{\mu_j(\hat y_l)} u_n^j$.
Let us denote by
$$
J_{h,k}^n(x, \bu)=\hat J_{h,k}^n(x, \bu,\ldots, \bu),
$$
then
\begin{equation}\label{prin_apen}
\left|J_{h,k}^n(x, \bu)-\hat J_{h,k}^n(x, \bu^1_n,\ldots, \bu^{n_s}_n)\right|\le C k,
\end{equation}
where the constant $C$ depends linearly on $L_u$ in \eqref{laLu}.
\end{lema}
\begin{proof}
Following \eqref{eq:funcional_fd}-\eqref{eq:eulerb} we have
\begin{eqnarray*}
J^{1,n}_{h,k}(x)&:=&J_{h,k}^n(x,\bu^1_n,\ldots,\bu^{n_s}_n)\\
&&\quad=h\sum_{j=n}^{N-1} \delta_h^{j-n}I_k L(\hat y_j,u_n^{1,j},\ldots,u_n^{n_s,j},t_j)
+I_kg(\hat y_{N})e^{-\lambda(T-t_n)}, \ \\
\hat y_{j+1}&=&\hat y_j+h I_k f(\hat y_j,u_n^{1,j},\ldots,u_n^{n_s,j},t_j),\ \hat y_n=x,\ j=n,\ldots {N}-1.\quad\label{eq:euler}
\end{eqnarray*}
And,
\begin{eqnarray*}
\hat J_{h,k}^2:=\hat J_{h,k}^n(x, \bu)&:=&h\sum_{j=n}^{N-1} \delta_h^{j-n} 
I_k L(\hat z_j,I_k  u_j,t_j)+I_kg(\hat z_{N})e^{-\lambda(T-t_n)},\\
\hat  z_{j+1}&=&\hat z_j+h I_k f(\hat z_j,I_k  u_j,t_j),\ \hat z_n=x, \ j=n,\ldots {N}-1,
\end{eqnarray*} 
where, for $\hat z_j=\sum_{l=1}^{n_s}\mu_l(\hat z_j)x_l$ then
\begin{eqnarray*}
I_k f(\hat z_j,I_k  u_j,t_j)&=&\sum_{l=1}^{n_s}\mu_l(\hat z_j)f(x_l,\sum_{m=1}^{n_s}\mu_m(\hat y_j) u^m_n,t_j),\\
I_k L(\hat z_j,I_k  u_j,t_j)&=&\sum_{l=1}^{n_s}\mu_l(\hat z_j)L(x_l,\sum_{m=1}^{n_s}\mu_m(\hat y_j) u^m_n,t_j).
\end{eqnarray*}
Let
$\tilde y(s)=\hat y_l$  and let $\tilde z(s)=\hat z_l$, $l=[s/h]$. Let $ \tilde u^j(s)=u_n^{j,l}$, $\tilde u(s)=u_l$, $s\in[lh,(l+1)h)$. Then,
\begin{eqnarray*}
\tilde y(s)&=&x+\int_{t_n}^{[s/h]h}I_kf(\tilde y(\tau), \tilde u^1(\tau),\ldots, \tilde u^{n_s}(\tau),[\tau/h]h) d\tau,\\
\tilde z(s)&=&x+\int_{t_n}^{[s/h]h}I_kf(\tilde z(\tau),I_k \tilde u(\tau),[\tau/h]h) d\tau.
\end{eqnarray*}
We then obtain
\begin{eqnarray*}
&&\|\tilde y(t)-\tilde z(t)\|_\infty\le\\
&&\quad \int_{t_n}^{[s/h]h}\|I_kf(\tilde y(\tau), \tilde u^1(\tau),\ldots, \tilde u^{n_s}(\tau),[\tau/h]h)
-I_kf(\tilde z(\tau),I_k \tilde u(\tau),[\tau/h]h) \|_\infty d\tau.
\end{eqnarray*}
We have to bound the difference of the two polynomials. For any $s\in[lh,(l+1)h)$ we will bound
\begin{eqnarray}\label{inter_plus}
\sum_{j=1}^{n_s}\mu_j(\hat y_l)f(x_j, u_n^j,t_l)-\sum_{j=1}^{n_s}\mu_j(\hat z_l)f(x_j,\sum_{m=1}^{n_s}\mu_m(\hat y_k) u_n^m,t_l).
\end{eqnarray}
We decompose \eqref{inter_plus} in two terms
\begin{eqnarray*}
I&=&\sum_{j=1}^{n_s}\mu_j(\hat y_k)f(x_j, u_n^j,t_l)-\sum_{j=1}^{n_s}\mu_j(\hat y_k)f(x_j,\sum_{m=1}^{n_s}\mu_l(\hat y_k) u_n^m,t_l),
\\
II&=&\sum_{j=1}^{n_s}\mu_j(\hat y_k)f(x_j,\sum_{m=1}^{n_s}\mu_l(\hat y_k) u_n^m,t_l)-
\sum_{j=1}^{n_s}\mu_j(\hat z_k)f(x_j,\sum_{m=1}^{n_s}\mu_l(\hat y_k) u_n^m,t_l).
\end{eqnarray*}
Let $\overline u(y)$ be the piecewise linear function satisfying
$$
\overline u(x_i)=u_n^i, \quad i=1,\ldots,n_s.
$$
We observe that $I$ is the difference of two interpolants. The term on the left-hand side is the interpolant  $I_k g(y)$ that interpolates the function $g(y)=f(y,\overline u(y),t_l)$ at $x_i$ for $i=1,\ldots,n_s$. The term on the right-hand side
considers the funtion $f(y,\overline u(y),t_l)$ and interpolates only in the first argument. We will denote this interpolant
by $I_k f(y,\overline u(y),t_l)$ such that for $y=\sum_{j=1}^{n_s}\mu_j(y) x_j$
$$
I_k f(y,\overline u(y),t_l)=\sum_{j=1}^{n_s}\mu_j(y)f(x_j,\overline u(y),t_l).
$$
We decompose $I$ in two terms
\begin{eqnarray*}
|I|\le |I_1|+|I_2|:=|I_kg(\hat y_k)-g(\hat y_k)|+|f(\hat y_k,\overline u(\hat y_k),t)-I_kf(\hat y_k,\overline u(\hat y_k),t)|.
\end{eqnarray*}
The first term can be bounded as the interpolation error in the function $g$. Following \cite{paper_lip}
the error is $O(k)$ with a constant that depends on the Lipschitz constant of $g$. Let us observe that
applying \eqref{lip_f_1}, \eqref{lip_fu}
\begin{eqnarray*}
&&\|g(y_1)-g(y_2)\|_2=\|f(y_1,\overline u(y_1),t_l)-f(y_2,\overline u(y_2),t_l)\|_2\\
&&\quad\le
\|f(y_1,\overline u(y_1),t_l)-f(y_1,\overline u(y_2),t_l)\|_2\\
&&\quad\  +\|f(y_1,\overline u(y_2),t_l)-f(y_2,\overline u(y_2),t)|\|_2\\
&&\quad\le L_f\left( \|\overline u(y_1)-\overline u(y_2)\|_2
+\|y_1-y_2\|_2\right).
\end{eqnarray*}
Since
\begin{equation}\label{new_as}
\|\overline u(y_1)-\overline u(y_2)\|_2\le L_{ u}\|y_1-y_2\|_2,
\end{equation}
for $L_u$ defined in \eqref{laLu} then $g$ has Lipschitz constant $L_f(L_u+1)$ and then
$$
|I_1|\le C(L_f,L_f(L_{u}+1)) k.
$$
To bound the second term we also apply \cite{paper_lip} and since the second argument is fixed the bound depends only on the
Lipschitz constant of $f$ so that
$$
|I_2|\le C(L_f) k.
$$
To bound $II$ we observe that we have the difference of two interpolants in the $y$ variable, the second argument is fixed, i.e., we interpolate the function $g(y)=f(y,u)$. The difference of two interpolants can be bounded
in terms of the values at which we interpolate and the Lipschitz constant of the function $g$. Then
$$
|II|\le \overline C(L_f)\|\hat y_k-\hat z_k\|_\infty=\overline C(L_f)\|\tilde y(s)-\tilde z(s)\|_\infty.
$$
Then, we reach
$$
\|\tilde y(s)-\tilde z(s)\|_\infty\le  \overline C(L_f)\int_{t_n}^{[s/h]h}\|\tilde y(\tau)-\tilde z(\tau)\|_\infty d\tau +C(L_f,L_{ u}) (T-t_n) k,
$$ 
and arguing as in \eqref{ymenosytilde} we conclude 
\begin{equation}\label{lay}
\|\tilde y(t)-\tilde z(t)\|_\infty\le Ck.
\end{equation}
To conclude we only need to bound the difference of the two functionals which can be easily obtained
arguing as in Lemma \ref{lema1} and applying the same arguments of this appendix to $L$ instead of $f$.
\end{proof}

  \end{document}